\documentclass[12pt]{amsart}
\usepackage{amsmath}
\usepackage{mathtools}
\usepackage{amssymb}
\usepackage{amsthm}
\usepackage{graphicx}
\usepackage{enumerate}
\usepackage[mathscr]{eucal}
\usepackage[pagewise]{lineno}
\usepackage{tikz}
\usetikzlibrary{decorations.text,calc,arrows.meta}
\theoremstyle{plain}

\numberwithin{equation}{section}
\usepackage[english]{babel}
\DeclareMathOperator{\conv}{conv}
\DeclareMathOperator{\sgn}{sgn}

\newtheorem{theorem}{Theorem}[subsection]

\newtheorem{lemma}[theorem]{Lemma}

\theoremstyle{definition}
\newtheorem{definition}[theorem]{Definition}

\usepackage{hyperref}
\hypersetup{
    colorlinks=true,
    linkcolor=blue,
    filecolor=magenta,      
    urlcolor=cyan,
}
\usepackage[pagewise]{lineno}
\usepackage{hyperref}

\begin{document}

\title[Birkhoff-James Orthogonality in Function Spaces]{Birkhoff-James Orthogonality and Its Pointwise Symmetry in Some Function Spaces}

\author[Bose]{Babhrubahan Bose}
                \newcommand{\acr}{\newline\indent}
                

\address[Bose]{Department of Mathematics\\ Indian Institute of Science\\ Bengaluru 560012\\ Karnataka \\INDIA\\ }
\email{babhrubahanb@iisc.ac.in}

 \thanks{The research of Babhrubahan Bose is funded by PMRF research fellowship under the supervision of Professor Apoorva Khare and Professor Gadadhar Misra.}

\begin{abstract}
    We study Birkhoff-James orthogonality and its pointwise symmetry in commutative $C^*$ algebras, i.e., the space of all continuous functions defined on a locally compact Hausdorff space which vanish at infinity. We use this characterization to obtain the characterization of Birkhoff-James orthogonality on $L_\infty$ space defined on any arbitrary measure space. We also do the same for the $L_p$ spaces for $1\leq p<\infty$.
\end{abstract}

\maketitle

\section*{Introduction}
{In recent times, symmetry of Birkhoff-James orthogonality has been a topic of considerable interest \cite{annal},  \cite{dkp}, \cite{1}, \cite{3},  \cite{4},  \cite{5}, \cite{8}. It is now well known that the said symmetry plays an important role in the study of the geometry of Banach spaces. The present article aims to explore Birkhoff-James orthogonality and its pointwise symmetry in some function spaces. We have completed such a study for some well studied sequence spaces, namely $\ell_p$ for $1\leq p\leq\infty$, $c$, $c_0$ and $c_{00}$ in \cite{usseq}. Here we take the study one step further by doing the same for commutative $C^*$ algebras and $L_p(X)$ for $1\leq p\leq\infty$ and any measure space $X$.}

Let us now establish the relevant notations and terminologies to be used throughout the article. Denote the scalar field $\mathbb{R}$ or $\mathbb{C}$ by $\mathbb{K}$ and recall the sign function $\sgn:\mathbb{K}\to\mathbb{K},$ given by
\[\sgn(x)=
\begin{cases}
\frac{x}{|x|},\;x\neq0,\\
0,\;x=0.
\end{cases}\] 
Consider a normed linear space $\mathbb{X}$ over $\mathbb{K}$ and denote its continuous dual by $\mathbb{X}^*$. Let $J(x)$ denote the collection of all support functionals of a non-zero vector $x\in \mathbb{X}$, i.e.,
\begin{align}\label{support}
    J(x):=\{f\in \mathbb{X}^*:\|f\|=1,\;|f(x)|=\|x\|\}.
\end{align}
A non-zero {element} $x\in\mathbb{X}$ is said to be \textit{smooth} if $J(x)$ is singleton.\par
Given $x,y\in \mathbb{X}$, $x$ is said to be \textit{Birkhoff-James orthogonal} to $y$ \cite{B}, denoted by $x\perp_By$, if
\begin{align*}
    \|x+\lambda y\|\geq\|x\|,~~\textit{for~all~}\lambda\in\mathbb{K}.
\end{align*}
James proved in \cite{james} that $x\perp_By$ if and only if $x=0$ or there exists $f\in J(x)$ such that $f(y)=0$. In the same article he also proved that a non-zero $x\in \mathbb{X}$ is smooth if and only if Birkhoff-James orthogonality is right additive at $x$, i.e.,
\begin{align*}
    x\perp_By,~x\perp_Bz~~\Rightarrow~~x\perp_B(y+z),~~\textit{for every}~y,z\in\mathbb{X}.
\end{align*}
\par
Birkhoff-James orthogonality is not symmetric in general, i.e., $x\perp_By$ does not necessarily imply that $y\perp_Bx$. In fact, James proved in \cite{james2} that Birkhoff-James orthogonality is symmetric in a normed linear space of dimension higher than 2 if and only if the space is an inner product space. However, the importance of studying the pointwise symmetry of Birkhoff-James orthogonality in describing the geometry of normed linear spaces has been illustrated in \cite[Theorem 2.11]{CSS}, \cite[Corollary 2.3.4.]{Sain}. Let us recall the following definition in this context from \cite{Sain2}, which will play an important part in our present study.
\begin{definition}
An element $x$ of a normed linear space $\mathbb{X}$ is said to be \textit{left-symmetric} (\textit{resp. right-symmetric}) if 
\begin{align*}
    x\perp_By\;\Rightarrow\; y\perp_Bx~~(\textit{resp.~}y\perp_Bx\;\Rightarrow\;x\perp_By),
\end{align*}
for every $y\in \mathbb{X}$.
\end{definition}
Note that by the term \textit{pointwise symmetry of Birkhoff-James orthogonality}, we refer to the left-symmetric and the right-symmetric points of a given normed linear space. The left-symmetric and the right-symmetric points of $\ell_p$ spaces where $1\leq p\leq \infty$, $p\neq2,$ were characterized in \cite{usseq}. {Here we generalize these results in $L_p(X)$ for any measure space $X$ and $p\in[1,\infty]\setminus\{2\}$.} For doing this generalization, we need to characterize Birkhoff-James orthogonality, smooth points, left symmetric points and right symmetric points in commutative $C^*$ algebras, i.e., $C_0(X)$, the space of all continuous functions vanishing at infinity defined on a locally compact Hausdorff space $X$. These characterizations in a given Banach space are important in understanding the geometry of the Banach space. We refer the readers to \cite{annal}, \cite{dkp}, \cite{1}, \cite{3}, \cite{4}, \cite{5}, \cite{8}, \cite{10}, \cite{SRBB}, \cite{12}, \cite{turnsek} for some prominent work in this direction. \par
In the first section we completely characterize Birkhoff-James orthogonality in commutative $C^*$ algebras, i.e., the space of all $\mathbb{K}$-valued continuous functions vanishing at infinity that are defined on a locally compact Hausdorff space $X$ and then characterize the left-symmetric and the right-symmetric points of the space.\par
In the second section, we use the results in the first section to completely characterize Birkhoff-James orthogonality, smoothness and pointwise symmetry of Birkhoff-James orthogonality in $L_\infty(X)$. It can be noted that we are establishing these results for an arbitrary measure space $X$ and in particular, we are not imposing any additional condition on $X$ such as finiteness or $\sigma$-finiteness of the measure.
In the third and fourth sections we {obtain} the same characterizations for $L_1(X)$ and $L_p(X)$ spaces $(p\in(1,\infty)\setminus\{2\})$. Observe that the $p=2$ case is trivial since $L_2(X)$ is a Hilbert space.

\section{Birkhoff-James orthogonality in commutative $C^*$ algebras}
The aim of this section is to obtain a necessary and sufficient condition for two elements in a commutative $C^*$ algebra to be Birkhoff-James orthogonal. Using that characterization, we characterize the smooth points and also study the pointwise symmetry of Birkhoff-James orthogonality in these algebras. We use the famous result Gelfand and Naimark proved in \cite{gelfand}, that any commutative $C^*$ algebra is isometrically $*$-isomorphic to $C_0(X)$ for some locally compact Hausdorff space $X$. Recall that $C_0(X)$ denotes the space of $\mathbb{K}$-valued continuous maps $f$ on $X$ such that
\[\lim\limits_{x\to\infty}f(x)=0,\]
equipped with the supremum norm, where $X\cup\{\infty\}$ is the one-point compactification of $X$. Also note that the $C^*$ algebra is unital if and only if $X$ is compact.\par
We also recall that by the Riesz representation theorem in measure theory, the continuous dual of $C_0(X)$ is isometrically isomorphic to the space of all regular complex finite Borel measures on $X$ equipped with total variation norm and the functional $\Psi_\mu$ corresponding to a measure $\mu$ acting by,
\begin{equation*}
    \Psi_\mu(f):=\int\limits_Xfd\mu,~~f\in C_0(X).
\end{equation*}
\subsection{Birkhoff-James orthogonality in $C_0(X)$}\hfill
\\

We begin with defining the \textit{norm attaining set} of an element $f\in C_0(X)$ by,
\[M_f:=\{x\in X:|f(x)|=\|f\|\}.\]
Clearly, $M_f$ is a compact subset of $X$. We state a characterization of the support functionals of an element $f\in C_0(X)$ using the norm attaining set. The proof of the result relies on elementary computations.
\begin{theorem}\label{norm}
Suppose $f\in C_0(X)$ and $f\neq0$. Let $\mu$ be a complex regular Borel measure. Then $\mu$ is of unit total variation corresponding to a support functional of $f$ if and only if $|\mu|\left(X\backslash M_f\right)=0$ and for almost every $x\in M_f$, with respect to the measure $\mu$, $d\mu(x)=\overline{\sgn(f(x))} d|\mu|(x)$.
\end{theorem}

We now come to the characterization of Birkhoff-James orthogonality in $C_0(X)$.
\begin{theorem}\label{ortho}
If $f,g\in C_0(X)$ and $f\neq0$, then $f\perp_Bg$ if and only if $0\in\conv\{\overline{f(x)}g(x):x\in M_f\}$.
\end{theorem}
\begin{proof}
Let $0\in\conv\{\overline{f(x)}g(x):x\in M_f\}$. Then there exist $n\in\mathbb{N}$, $\lambda_1,\lambda_2,\dots,\lambda_n\geq0$ with $\sum\limits_{k=1}^n\lambda_k=1$ and\\
\begin{equation*}
    0=\sum\limits_{k=1}^n\lambda_k\overline{f(x_k)}g(x_k),\\
\end{equation*}
for some $x_1,x_2,\dots, x_n\in M_f$. Consider the functional\\
\begin{equation*}
    \Psi:h\mapsto\frac{1}{\|f\|}\sum\limits_{k=1}^n\lambda_k\overline{f(x_k)}h(x_k),~~h\in C_0(X).\\
\end{equation*}
Then for $h\in C_0(X)$,\\
\begin{equation*}
    |\Psi(h)|=\left|\frac{1}{\|f\|}\sum\limits_{k=1}^n\lambda_k\overline{f(x_k)}h(x_k)\right|\leq\|h\|\left(\sum\limits_{k=1}^n\lambda_k\right)=\|h\|.\\
\end{equation*}
Also,\\
\begin{equation*}
    \Psi(f)=\frac{1}{\|f\|}\sum\limits_{k=1}^n\lambda_k\overline{f(x_k)}f(x_k)=\|f\|\left(\sum\limits_{k=1}^n\lambda_k\right)=\|f\|,\\
\end{equation*}
and\\
\begin{equation*}
    \Psi(g)=\frac{1}{\|f\|}\sum\limits_{k=1}^n\lambda_k\overline{f(x_k)}g(x_k)=0.\\
\end{equation*}
Hence $\Psi$ is a support functional of $f$ such that $\Psi(g)=0$, giving $f\perp_Bg$ and proving the sufficiency. \par
Conversely, suppose $f\perp_Bg$. Then there is a support functional of $f$ that annihilates $g$. Invoking Theorem \ref{norm} we obtain a complex regular Borel measure $\nu$ having $|\nu|(M_f)=1$ and 
\begin{equation*}
    \int\limits_Xhd\nu=\int\limits_{M_f}h(x)\overline{\sgn(f(x))}d|\nu|(x),~~\textit{for every}~h\in C_0(X),
\end{equation*}
such that 
\begin{equation*}
    0=\int\limits_Xgd\nu=\int\limits_{M_f}g(x)\frac{\overline{f(x)}}{\|f\|}d|\nu|(x).
\end{equation*}
Suppose $\Lambda$ is the space of all positive semi-definite regular Borel probability measures on $M_f$ and $\Phi:\Lambda\to\mathbb{K}$ given by,
\begin{equation*}
    \Phi(\mu):=\int\limits_{M_f}\overline{f(x)}g(x)d\mu(x),~~\mu\in\Lambda.
\end{equation*}
Observe that since $\Lambda$ is convex, so is $\Phi(\Lambda)$. Also, as $\Lambda$ is the collection of all support functionals of $|f|\in C_0(X)$, it is compact under the weak* topology by the Banach-Alaoglu theorem \cite[subsection 3.15, p.68]{BAT}. Now, the map $\Phi$ is evaluation at the element $\overline{f}g\in C_0(X)$ on $\Lambda$ and hence is continuous where $\Lambda$ is equipped with the weak* topology. Therefore, $\Phi(\Lambda)$ is compact and hence by the Krein-Milman theorem \cite{KMT}, 
\begin{equation*}
    \Phi(\Lambda)=\overline{\conv}\{\lambda:\lambda~\textit{is~an~extreme~point~of~}\Phi(\Lambda)\}.
\end{equation*}\par
We claim that any extreme point of $\Phi(\Lambda)$ is of the form $\overline{f(x)}g(x)$ for some $x\in M_f$. Suppose, on the contrary, $\Phi(\mu)$ is an extreme point of $\Phi(\Lambda)$ and $\mu$ is not a Dirac delta measure. If $\overline{f}g$ is constant on the support of $\mu$, clearly, $\Phi(\mu)=\overline{f(x)}g(x)$ for any $x$ in the support of $\mu$. Otherwise, there exist $x,y$ in the support of $\mu$ such that $\overline{f(x)}g(x)\neq \overline{f(y)}g(y)$. Consider $0<\delta<\frac{1}{2}|\overline{f(x)}g(x)-\overline{f(y)}g(y)|$ and $U_x\subset M_f$ open such that 
\begin{align*}
    z\in U_x~\Rightarrow~|\overline{f(x)}g(x)-\overline{f(z)}g(z)|<\delta.
\end{align*}
Then $U_x$ and $M_f\backslash U_x$ are two disjoint subsets of $M_f$ having non-zero measures since $M_f\backslash U_x$ contains an open subset of $M_f$ containing $y$. Clearly, since $\mu$ can be written as a convex combination of $\frac{1}{\mu\left(U_x\right)}\mu|_{U_x}$ and $\frac{1}{\mu\left(M_f\backslash U_x\right)}\mu|_{M_f\backslash U_x}$, we get
\begin{align*}
    \Phi(\mu)=\frac{1}{\mu(U_x)}\int\limits_{U_x} \overline{f(z)}g(z)d\mu(z).
\end{align*}
Hence, we have
\begin{align*}
    \left|\overline{f(x)}g(x)-\Phi(\mu)\right|&=\left|\overline{f(x)}g(x)-\frac{1}{\mu(U_x)}\int\limits_{U_x} \overline{f(z)}g(z)d\mu(z)\right|\\
    &\leq\frac{1}{\mu\left(U_x\right)}\int\limits_{U_x}|\overline{f(x)}g(x)-\overline{f(z)}g(z)|d\mu(z)\leq\delta.
\end{align*}
Since $0<\delta<\frac{1}{2}|\overline{f(x)}g(x)-\overline{f(y)}g(y)|$ is arbitrary, we obtain that $\Phi(\mu)=\overline{f(x)}g(x)$ establishing our claim.\par
Hence, 
\begin{equation}\label{convex}
    0=\Phi(|\nu|)\in\Phi(\Lambda)=\overline{\conv}\{\overline{f(x)}g(x):x\in M_f\}.
\end{equation}\par
 We now prove that if $K\subset\mathbb{K}$ is compact, $\conv(K)=\overline{\conv}(K)$. Suppose $z$ is a limit point of $\conv(K)$. Then there exists a sequence of elements $z_n$ in $\conv(K)$ converging to $z$. But by Caratheodory's theorem \cite{caratheodory}, for every $n\in\mathbb{N}$, there exist $\lambda_i^{(n)}\in[0,1]$ and $z_i^{(n)}\in K$ for $i=1,2,3$ such that 
\begin{equation*}
    \sum\limits_{i=1}^3\lambda_i^{(n)}=1,~~\sum\limits_{i=1}^3\lambda_i^{(n)}z_i^{(n)}=z_n.
\end{equation*}
Since $[0,1]$ and $K$ are both compact, we may consider an increasing sequence of natural numbers $\left(n_k\right)_{k\in\mathbb{N}}$ such that $\{\lambda_1^{(n_k)}\}_{k\in\mathbb{N}}$, $\{\lambda_2^{(n_k)}\}_{k\in\mathbb{N}}$, $\{\lambda_3^{(n_k)}\}_{n_k\in\mathbb{N}}$, $\{z_1^{(n_k)}\}_{k\in\mathbb{N}}$, $\{z_2^{(n_k)}\}_{k\in\mathbb{N}}$ and $\{z_3^{(n_k)}\}_{k\in\mathbb{N}}$ are all convergent and thereby obtain that $z\in\conv(K)$. \par
As $M_f$ is compact, $\{\overline{f(x)}g(x):x\in M_f\}$ is a compact subset of $\mathbb{K}$ and hence by \eqref{convex},
\begin{equation*}
    0\in\conv\{\overline{f(x)}g(x):x\in M_f\},
\end{equation*}
establishing the necessity.
\end{proof}
We now furnish a generalization of \cite[Corollary 2.2]{function} characterizing the smoothness of an element of $C_0(X)$.
\begin{theorem}\label{smooth}
 A point $f\in C_0(X)$ is smooth if and only if $M_f$ is a singleton set.
 \end{theorem}
 \begin{proof}
 First if $M_f$ is a singleton set, say $\{x_0\}$,  then clearly by Theorem \ref{ortho}, $f\perp_Bg$ for $g\in C_0(X)$ if and only if $g(x_0)=0$. Hence clearly, for $g,h\in C_0(X)$, $f\perp_B g,h$ would imply
 \begin{equation*}
     g(x_0)=h(x_0)=0~\Rightarrow~ g(x_0)+h(x_0)=0~\Rightarrow ~f\perp_B(g+h).
 \end{equation*}
 Hence $f$ is smooth.\par
 Conversely, if $x_1,x_2\in M_f$, $x_1\neq x_2$, then $\Psi_1,\Psi_2:C_0(X)\to\mathbb{K}$ given by
 \begin{equation*}
     \Psi_i(g):=\overline{\sgn(f(x_i))}g(x_i),~~g\in C_0(X),~i=1,2,
 \end{equation*}
 are two support functionals of $f$. Now, since $X$ is Hausdorff, there exists $U\subset X$ open such that $x_1\in U$ and $x_2\notin U$. Hence, there exists a continuous map $h$ on $X$ having compact support, vanishing outside $U$ and $h(x_1)=1$. Thus $h\in C_0(X)$ and $\Psi_1(h)\neq\Psi_2(h)$.  Therefore $f$ is not smooth.
 \end{proof}
 \subsection{Pointwise symmetry of Birkhoff-James orthogonality in $C_0(X)$}\hfill
 \\
 
 In this subsection we characterize the pointwise symmetry of Birkhoff-James orthogonality in $C_0(X)$.
 We begin with our characterization of the left symmetric points of $C_0(X)$.
 \begin{theorem}\label{left}
  An element $f\in C_0(X)$ is a left symmetric point of $C_0(X)$ if and only if $f$ is identically zero or $M_f$ is singleton and $f$ vanishes outside $M_f$.
 \end{theorem}
 \begin{proof}
 We begin with the sufficiency. If $M_f=\{x_0\}$ for some $x_0\in X$, then by Theorem \ref{ortho}, $f\perp_Bg$ for any $g\in C_0(X)$ if and only if $g(x_0)=0$. Then clearly, $x_0\notin M_g$ and hence if $x_1\in M_g$, $f(x_1)=0$ giving $g\perp_Bf$ by Theorem \ref{ortho}.\par
 Conversely, suppose $f\in C_0(X)$ is left-symmetric and not identically zero. Suppose $x_1\in M_f$ and $x_2\in X$ such that $x_1\neq x_2$, $f(x_2)\neq0$. Consider $U,U'\subset X$ open containing $x_2$ such that $x_1\notin U$ and $f$ does not vanish on $U'$. Set $U''=U\cap U'$. Consider a continuous function $h:X\to[0,1]$ having compact support such that $h(x_2)=1$ and $h$ vanishes outside $U''$. Set $g(x):=\sgn(f(x))h(x)$, $x\in X$. Then clearly $g\in C_0(X)$ and $g(x_1)=0$ giving $f\perp_Bg$ by Theorem \ref{ortho}. But clearly $\sgn(g(x))=\sgn(f(x))$ for every $x\in M_g$. Hence $g\not\perp_Bf$ by Theorem \ref{ortho}, establishing the necessity.
 \end{proof}
 Note that this theorem clearly states that if $X$ has no singleton connected component, $C_0(X)$ has no non-zero left symmetric point. \par
 We next characterize the right-symmetric points.
 \begin{theorem}\label{right}
  An element $f\in C_0(X)$ is right-symmetric if and only if $M_f=X$. Hence, in particular if $X$ is not compact, $C_0(X)$ has no non-zero right symmetric point.
 \end{theorem}
 \begin{proof}
 We again begin with the sufficiency. If $M_f=X$, then by Theorem \ref{ortho}, $g\perp_B f$ only if
 \begin{align*}
     0\in\conv\{\overline{g(x)}f(x):x\in M_g\}.
 \end{align*}
 Since $M_g\subset X=M_f$, we clearly obtain
 \begin{align*}
     0\in\conv\{\overline{f(x)}g(x):x\in M_f\},
 \end{align*}
 and hence $f\perp_Bg$ by Theorem \ref{ortho}.\par
 Conversely, suppose $f\in C_0(X)$ is right-symmetric and not identically zero. For the sake of contradiction, let us assume that $M_f\neq X$. We consider two cases. \\
 \textbf{Case 1:} \textit{$f(x_0)=0$ for some $x_0\in X$.}\\
 Since $x_0\notin M_f$ and $M_f$ is compact, we obtain $U,U'\subset X$ open such that $U\cap U'=\emptyset$, $x_0\in U$ and $M_f\subset U'$. Now, consider two continuous functions $h,h':X\to[0,1]$ having compact supports such that $h(x_0)=1$, $h$ vanishes outside $U$ and $h'$ is identically 1 on $M_f$ and vanishes outside $U'$. Set $g(x):=\|f\|h(x)+f(x)h'(x)$, $x\in X$. Then $\|g\|=\|f\|$ and $x_0\in M_g$. Hence by Theorem \ref{ortho}, $g\perp_Bf$. However, if $x\in M_f$, $g(x)=f(x)$ and hence by Theorem \ref{ortho}, $f\not\perp_Bg$.\\
 \textbf{Case 2:} \textit{$f$ is non-zero everywhere on $X$ but there exists $x_0\in X\backslash M_f$.}\\
 Let us again consider $U$, $U'$, $h$ and $h'$ as before and set $g(x):=-\|f\|\sgn(f(x))h(x)+f(x)h'(x)$. Then clearly, $M_f\subset M_g$ and $x_0\in M_g$. Also, $g(x)=f(x)$ for $x\in M_f$ and $g(x_0)=-\|f\|\sgn(f(x_0))$ giving $g\perp_Bf$ by Theorem \ref{ortho}. Also, by the same theorem $f\not\perp_Bg$, proving the necessity. 
 \end{proof}

\section{Birkhoff-James orthogonality and its pointwise symmetry in $L_\infty$ spaces}
In this section, we study Birkhoff-James orthogonality and its pointwise symmetry in $L_\infty$ spaces. Note that since $L_\infty$ spaces are also commutative $C^*$ algebras, we are going to use the results from Section 2 for this study. We begin with representing $L_\infty(X)$ as $C_0(Y)$ for some suitable locally compact, Hausdorff $Y$. We then study this representation and use the results of Section 2 to characterize Birkhoff-James orthogonality and its pointwise symmetry in $L_\infty(X)$.\par

We begin by considering a positive measure space $(X,\Sigma,\lambda)$ and $L_\infty^\mathbb{K}(X,\Sigma,\lambda)$, the space of all essentially bounded $\mathbb{K}$ valued functions on $X$ equipped with the essential supremum norm. Without any ambiguity, we refer to $L_\infty^\mathbb{K}(X,\Sigma,\lambda)$ as $L_\infty(X)$. \par
We now represent $L_\infty(X)$ as the space of continuous functions on a compact topological space equipped with the supremum norm. We begin with a definition:
\begin{definition}
    A \textit{0-1 measure with respect to $\lambda$} is a finitely additive set function $\mu$ on $(X,\Sigma)$ taking values in $\{0,1\}$ such that $\mu(X)=1$, $\mu(A)=0$ whenever $\lambda(A)=0$.
\end{definition}
Let us define $\mathfrak{G}$ as the collection of all 0-1 measures with respect to $\lambda$. We consider the $t$ topology on $\mathfrak{G}$ having a basis consisting of sets of the following form:
\begin{equation*}
    t(A):=\{\mu\in\mathfrak{G}:\mu(A)=1\},~~A\in\Sigma,~\lambda(A)>0.
\end{equation*}
Yosida and Hewitt proved the following representation result in \cite{Y-H}:
\begin{theorem}\label{representation}
\strut\\
    1. The topological space $\left(\mathfrak{G},t\right)$ is compact and Hausdorff.\\
    2. The map $T:L_\infty(X)\to C^\mathbb{K}\left(\mathfrak{G},t\right)$ given by
    \begin{equation*}
        T(f)(\mu):=\int\limits_Xfd\mu,~~\mu\in\mathfrak{G},~f\in L_\infty(X),
    \end{equation*}
    is an isometric isomorphism.
\end{theorem}
In the first subsection we study the space of 0-1 measures with respect to $\lambda$ and integrals with respect to the measures. We characterize Birkhoff-James orthogonality between two elements of $L_\infty(X)$ in the second subsection along with characterization of smoothness of a point. The third subsection comprises of characterizations of pointwise symmetry in $L_\infty(X)$.
\subsection{0-1 measures with respect to $\lambda$ and $\lambda$-ultrafilters}\hfill
\\

In this subsection, we obtain a one to one correspondence between all the 0-1 measures with respect to $\lambda$ and all $\lambda$-ultrafilters and therefore use the $\lambda$-ultrafilters to study integrals with respect to the 0-1 measures with respect to $\lambda$. We begin with the definition of a $\lambda$-filter.
\begin{definition}
    A non-empty subset $\mathcal{F}$ of $\Sigma$ is called a \textit{$\lambda$-filter on $X$} if\\
    1. $\lambda(A)>0$ for every $A\in \mathcal{F}$.\\
    2. For every $A,B\in\mathcal{F}$, $A\cap B\in\mathcal{F}$.\\
    3. $B\in\mathcal{F}$ for every $B\supset A$, $B\in\Sigma$ and $A\in\mathcal{F}$\\
    A $\lambda$-filter $\mathcal{U}$ is called a \textit{$\lambda$-ultrafilter} if any $\lambda$-filter containing $\mathcal{U}$ is $\mathcal{U}$ itself.
\end{definition}
The existence of $\lambda$-ultrafilters is a direct consequence of Zorn's lemma. Before proceeding further, we derive a lemma that is going to be used throughout the section.
\begin{lemma}\label{significant}
Suppose $\mathcal{U}$ is a $\lambda$-ultrafilter.\\
1. If $A\in\Sigma$ such that $A\notin\mathcal{U}$, then there exists $B\in\mathcal{U}$, such that $\lambda(A\cap B)=0$.\\
2. If $\bigcup\limits_{k=1}^nA_k\in\mathcal{U}$ for $A_1, A_2,\dots,A_n\in\Sigma$, then $A_i\in\mathcal{U}$ for some $1\leq i\leq n$.
\end{lemma} 
\begin{proof}
1.If no such $B$ exists, then $\mathcal{F}:=\{A\cap C:C\in\mathcal{U}\}\cup\mathcal{U}$
is a $\lambda$-filter properly containing $\mathcal{U}$.\\
2. If $A_k\notin\mathcal{U}$ for every $1\leq k\leq n$, then by part 1, there exist $B_k\in\mathcal{U}$ for $1\leq k\leq n$ such that $\lambda(A_k\cap B_k)=0$. But then setting $B:=\bigcap\limits_{k=1}^n B_k$, we get
\begin{equation*}
    \lambda\left(\left(\bigcup\limits_{k=1}^nA_k\right)\cap B\right)\leq\sum\limits_{k=1}^n\lambda(A_k\cap B)=0,
\end{equation*}
violating the closure of $\mathcal{U}$ under finite intersections.
\end{proof}
\par
Let $\mathfrak{F}$ denote the collection of all $\lambda$-ultrafilters on $X$.
We now define the $\lambda$-ultrafilter corresponding to a 0-1 measure with respect to $\lambda$ and vice versa. Suppose $\mu\in\mathfrak{G}$. We define $\mathcal{U}_\mu$ by:
\begin{equation*}
    \mathcal{U}_\mu:=\{A\in\Sigma:\mu(A)=1\}.
\end{equation*}
Also for any $\lambda$-ultrafilter $\mathcal{U}$, let us define a set function $\mu^\mathcal{U}$ on $\Sigma$ by:
\begin{equation*}
    \mu^\mathcal{U}(A):=
    \begin{cases}
        1,~~A\in\mathcal{U},\\
        0,~~A\notin\mathcal{U}.
    \end{cases}
\end{equation*}
\begin{theorem}\label{representation2}
\strut\\
    1.For any $\mu\in\mathfrak{G}$, $\mathcal{U}_\mu\in\mathfrak{F}$ and is called the $\lambda$-ultrafilter corresponding to $\mu$. \\
    2. For any $\mathcal{U}\in\mathfrak{F}$, $\mu^\mathcal{U}\in\mathfrak{G}$ and is called the 0-1 measure with respect to $\lambda$ corresponding to $\mathcal{U}$.\\
    3. The maps $\mu\mapsto\mathcal{U}_\mu$, $\mu\in\mathfrak{G}$ and $\mathcal{U}\mapsto\mu^\mathcal{U}$, $\mathcal{U}\in\mathfrak{F}$ are inverse to each other and thereby establish a one to one correspondence between $\mathfrak{F}$ and $\mathfrak{G}$.
\end{theorem}
\begin{proof}
1. It is easy to verify that for any $\mu\in\mathfrak{G}$, $\mathcal{U}_\mu$ does not contain any set having $\lambda$-measure zero. Further, since $\mu$ is a 0-1 measure, $\mu(A)=1$ and $B\supset A$, $B\in\Sigma$ forces $\mu(B)=1$. Now, if $\mu(A)=\mu(B)=1$, then $\mu(A\cup B)=1$ since $\mu$ is a 0-1 measure and hence as $\mu$ is additive, 
\begin{equation*}
    1=\mu(A\cup B)=\mu(A)+\mu(B\backslash A)=1+\mu(B\backslash A).
\end{equation*}
Hence $\mu(B\backslash A)=0$ and so by additivity of $\mu$,
\begin{equation*}
    \mu(A\cap B)=\mu(A\cap B)+\mu(B\backslash A)=\mu(B)=1,
\end{equation*}
giving $A\cap B\in\mathcal{U}_\mu$. Now if $\mathcal{F}$ is another $\lambda$-filter containing $\mathcal{U}_\mu$, then consider $C\in\mathcal{F}\backslash\mathcal{U}_\mu$. Clearly, $\mu(C)=0$. Then $\mu(X\backslash C)=1$ and hence $X\backslash C\in\mathcal{U}_\mu\subset\mathcal{F}$ violating, that $\mathcal{F}$ is a $\lambda$-filter.\\
2. Clearly, $\mu^\mathcal{U}(A)=0$ for $\lambda(A)=0$ since $A\notin\mathcal{U}$. Now, if $A, B\in\Sigma$ and $A\cap B=\emptyset$, either exactly one of $A$ and $B$ is in $\mathcal{U}$ in which case $A\sqcup B\in\mathcal{U}$, or neither $A$ nor $B$ is in $\mathcal{U}$ in which case by Lemma \ref{significant}, $A\sqcup B\notin\mathcal{U}$. Clearly, in both cases,
\begin{equation*}
    \mu^\mathcal{U}(A\sqcup B)=\mu^\mathcal{U}(A)+\mu^\mathcal{U}(B).
\end{equation*} 
Hence $\mu^\mathcal{U}$ is a 0-1 measure with respect to $\lambda$.\\
3. This part is an easy verification.
\end{proof}
We now study the integrals under 0-1 measures with respect to $\lambda$ in the light of this one to one correspondence. We introduce two new definitions.
\begin{definition}
    A non-empty subset $\mathcal{B}$ of $\Sigma$ is said to be a \textit{$\lambda$-filter base} if\\
    1. $\lambda(A)>0$ for every $A\in\mathcal{B}$.\\
    2. For every $A,B\in\mathcal{B}$, there exists $C\in\mathcal{B}$ such that $C\subset A\cap B$.
\end{definition}
Any $\lambda$-filter base $\mathcal{B}$ is contained in a unique minimal $\lambda$-filter given by 
\begin{equation*}
    \{B\in\Sigma:B\supset A~for~some~A\in\mathcal{B}\}.
\end{equation*}
Since every $\lambda$-filter is contained in a $\lambda$-ultrafilter (a direct application of Zorn's lemma), every $\lambda$-filter base is contained in a $\lambda$-ultrafilter.
We now define limit under a $\lambda$-filter.
\begin{definition}
    Suppose $\mathcal{F}$ is a $\lambda$-filter on $X$ and $f:X\to\mathbb{K}$ is a measurable function. Then the \textit{limit of the map $f$ under the $\lambda$-filter $\mathcal{F}$} (written as $\lim\limits_\mathcal{F}f$) is defined as $z_0\in\mathbb{K}$ if 
    \begin{equation*}
        \{x\in X:|f(x)-z_0|<\epsilon\}\in\mathcal{F}~~for~every~\epsilon>0.
    \end{equation*}
\end{definition}
We state a few elementary results pertaining to the limit of a measurable function under a $\lambda$-filter. We omit the proofs since the results follow directly from the definition of the limit.
\begin{theorem}
    Suppose $f:X\to\mathbb{K}$ is a measurable function and $\mathcal{F}$ is a $\lambda$-filter on $X$.\\
    1. $\lim\limits_{\mathcal{F}}f$ if exists is unique.\\
    2. If $g:\mathbb{K}\to\mathbb{K}$ is continuous, $\lim\limits_{\mathcal{F}}g\circ f=g\left(\lim\limits_\mathcal{F}f\right)$.\\
    3. Limits under a $\lambda$-filter respect addition, multiplication, division and multiplication with a constant.\\
\end{theorem}
We now come to our second key result of this subsection.
\begin{theorem}\label{integral representation}
    Suppose $f\in L_\infty(X)$.\\
    1. If $\mathcal{U}$ is a $\lambda$-ultrafilter, $\lim\limits_\mathcal{U}f$ exists and is well defined.\\
    2. If $\mu\in\mathfrak{G}$, 
    \begin{equation*}
        \lim\limits_{\mathcal{U}_\mu}f=\int\limits_Xfd\mu.
    \end{equation*}
\end{theorem}
\begin{proof}
1. Clearly $\lim\limits_\mathcal{U}f$, if exists, must lie in the set $\mathcal{D}:=\{z\in\mathbb{K}:|z|\leq\|f\|_\infty\}$. Now, suppose $\lim\limits_\mathcal{U}f$ does not exist. Then for every $z\in\mathcal{D}$, there exists $\epsilon_z>0$ such that 
\begin{equation*}
    \{x\in X:|f(x)-z|<\epsilon_z\}\notin\mathcal{U}.
\end{equation*}
Set $B_z:=\{w\in\mathbb{K}:|w-z|<\epsilon_z\}$. Then $\{B_z:z\in\mathcal{D}\}$ is an open cover of the compact set $\mathcal{D}$ and therefore must have a finite sub-cover, say $\{B_{z_1},B_{z_2},\dots,B_{z_n}\}$. We further set 
\begin{equation*}
    A_i:=\{x\in X: |f(x)-z_i|<\epsilon_{z_i}\},~~1\leq i\leq n.
\end{equation*}
Hence clearly, $A_i\notin\mathcal{U}$ and 
\begin{equation*}
    \bigcup\limits_{i=1}^nA_i=X\backslash B,~for~some~B\in\Sigma,~\lambda(B)=0.
\end{equation*}
Thus by Lemma \ref{significant}, $X\backslash B\notin\mathcal{U}$. But then as $B\notin\mathcal{U}$, we arrive at a contradiction.\par 
In order to prove that the limit is well defined, consider $f$ and $f'$ essentially bounded such that $f=f'$ almost everywhere on $X$ with respect to $\lambda$. Now, if $\lim\limits_\mathcal{U}f=z_0$, then for any $\epsilon>0$, 
\begin{equation*}
    \{x\in X:|f(x)-z_0|<\epsilon\}\subset\{x\in X:|f'(x)-z_0|<\epsilon\}\cup\{x\in X: f(x)\neq f'(x)\}.
\end{equation*}
Now by Lemma \ref{significant}, $\{x\in X:|f'(x)-z_0|<\epsilon\}\in\mathcal{U}$ since $\lambda\left(\{x\in X: f(x)\neq f'(x)\}\right)=0$.\\
2. Suppose $\lim\limits_{\mathcal{U}_\mu}f=z_0$. Then for every $\epsilon>0$, 
\begin{equation*}
    \mu\left(\{x\in X:|f(x)-z_0|<\epsilon\}\right)=1,
\end{equation*}
and therefore
\begin{equation*}
    \mu\left(\{x\in X:|f(x)-z_0|\geq\epsilon\}\right)=0.
\end{equation*}
Hence we obtain that
\begin{align*}
    \left|\int\limits_Xfd\mu-z_0\right| \leq \int\limits_X|f-z_0|d\mu=\int\limits_{|f-z_0|<\epsilon}|f-z_0|d\mu
    \leq\epsilon.
\end{align*}
Since $\epsilon$ is arbitrary, 
\begin{equation*}
    \int\limits_Xfd\mu=z_0.
\end{equation*}

\end{proof}

We now establish a result that gives a collection of possible values of limits of an essentially bounded function under a $\lambda$-ultrafilter. 
\begin{theorem}\label{integral and limit}
    Suppose $f,g\in L_\infty(X)$. \\
    1. For any $z_0\in\mathbb{K}$, there exists a $\lambda$-ultrafilter $\mathcal{U}$ such that $\lim\limits_\mathcal{U}f=c_0$ if and only if for every $\epsilon>0$, 
    \begin{equation*}
        \lambda\left(\{x\in X:|f(x)-z_0|<\epsilon\}\right)>0.
    \end{equation*}
    2. For any $z_0,w_0\in\mathbb{K}$, there exists a $\lambda$-ultrafilter $\mathcal{U}$ such that $\lim\limits_\mathcal{U}f=z_0$ and $\lim\limits_\mathcal{U}g=w_0$ if and only if for any $\epsilon>0$, 
    \begin{equation*}
        \lambda\left(\{x\in X:|f(x)-z_0|<\epsilon,|g(x)-w_0|<\epsilon\}\right)>0.
    \end{equation*}
\end{theorem}
\begin{proof}
The necessary part of both the statements are clear. We therefore prove the sufficiency in the two statements.\\
1. Consider $\mathcal{B}\subset\Sigma$ given by
\begin{equation*}
    \mathcal{B}:=\left\{\{x\in X:|f(x)-z_0|<\epsilon\},~~\epsilon>0\right\}.
\end{equation*}
Clearly, $\mathcal{B}$ is a $\lambda$-filter base and hence there exists a $\lambda$-ultrafilter $\mathcal{U}$ containing $\mathcal{B}$. Clearly, by construction, $\lim\limits_\mathcal{U}f=z_0$.\\
2. Again consider $\mathcal{B}'\subset\Sigma$ given by
\begin{equation*}
    \mathcal{B}':=\left\{\{x\in X:|f(x)-z_0|<\epsilon\}\cap\{x\in X:|g(x)-w_0|<\epsilon\},~~\epsilon>0\right\}.
\end{equation*}
Clearly, $\mathcal{B}'$ too is a $\lambda$-filter base and hence there exists a $\lambda$-ultrafilter $\mathcal{U}'$ containing $\mathcal{B}'$. Also, by construction, clearly, $\lim\limits_{\mathcal{U}'}f=z_0$ and $\lim\limits_{\mathcal{U}'}g=w_0$.
\end{proof}

\subsection{Birkhoff-James orthogonality in $L_\infty(X)$}\hfill
\\

In this subsection we characterize Birkhoff-James orthogonality between two elements of $L_\infty(X)$ and use the characterization to study the smoothness of a point in $L_\infty(X)$. \par
The following characterization of orthogonality follows from Theorems \ref{ortho}, \ref{representation}, \ref{integral representation} and \ref{integral and limit}.
\begin{theorem}
    Suppose $f,g\in L_\infty(X)$ are non-zero. Then $f\perp_Bg$ if and only if
    \[0\in\conv\left\{z:\lambda\left(\{x\in X: |f(x)|>\|f\|_\infty-\epsilon,~|\overline{f(x)}g(x)-z|<\epsilon\} \right) >0~~\forall~\epsilon>0\right\}.\]
\end{theorem}
We now come to the characterization of smooth points in $L_\infty(X)$, but before this result, we prove a preliminary lemma.
\begin{lemma}\label{partition}
    If $f\in L_\infty(X)$ and 
    \begin{equation*}
        \lambda\left(\{x\in X:|f(x)|=\|f\|_\infty\}\right)=0,
    \end{equation*}
    there exist $A,B\subset X$ such that $A\cap B=\emptyset$ and 
    \begin{equation*}
        \lambda\left(\{x\in A: \|f\|-|f(x)|<\epsilon\}\right),~\lambda\left(\{x\in B: \|f\|_\infty-|f(x)|<\epsilon\}\right)>0,
    \end{equation*}
    for every $\epsilon>0$.
\end{lemma}
\begin{proof}
Since $\lambda\left(\{x\in X:|f(x)|=\|f\|_\infty\right\})=0$, clearly either, \[\lambda\left(\{x\in X:\|f\|_\infty-|f(x)|<\epsilon\}\right)<\infty,\] 
for some $\epsilon>0$ and 
\[\lambda\left(\{x\in X:\|f\|_\infty-|f(x)|<\delta\}\right)\to0,\] 
as $\delta\to0$ or, 
\[\lambda\left(\{x\in X:\|f\|_\infty-|f(x)|<\epsilon\}\right)=\infty\] 
for every $\epsilon>0$.\\
\textbf{Case 1:}
\textit{$\lambda\left(\{x\in X:\|f\|_\infty-|f(x)|<\epsilon_0\}\right)<\infty$ for some $\epsilon_0>0$.}\\
For $n\in\mathbb{N}$, we set $C_n=\{x\in X:\|f\|_\infty-|f(x)|<\epsilon_{n-1}\}$ and $D_n\subset C_n$ such that
\begin{equation*}
    0<\lambda(D_n)\leq\frac{1}{2}\lambda(C_n).
\end{equation*}
We further consider $\epsilon_n>0$ such that
\begin{equation*}
    \lambda\left(\{x\in X:\|f\|_\infty-|f(x)|<\epsilon_n\}\right)<\frac{1}{3}\lambda(D_n).
\end{equation*}
Finally, define
\begin{equation*}
    A_n:=D_n\backslash C_{n+1},~~A:=\bigsqcup\limits_{n=1}^\infty A_n,~~B:=X\backslash A.
\end{equation*}
Then clearly, since $\lambda(A_n)>0$ for every $n\in\mathbb{N}$ and $\lim\limits_{n\to\infty}\epsilon_n=0$, we have
\begin{equation*}
    \lambda\left(\{x\in A: \|f\|-|f(x)|<\epsilon\}\right)>0,
\end{equation*}
for every $\epsilon>0$. Again for any $\epsilon>0$, there exists $n\in\mathbb{N}$ such that $\epsilon>\epsilon_{n-1}$. Observe that
\begin{align*}
    \lambda\left(\{x\in B: \|f\|-|f(x)|<\epsilon\}\right)&\geq\lambda\left(\{x\in B: \|f\|-|f(x)|<\epsilon_{n-1}\}\right)\\
    &=\lambda\left(C_n\backslash\left(\bigsqcup\limits_{k\geq n}A_k\right)\right)\\
    &=\lambda(C_n)-\sum\limits_{n\geq k}\lambda(A_k)\\
    &\geq \lambda(C_n)-\sum\limits_{k\geq n}\lambda(D_k)\\
    &\geq \lambda(C_n)-\sum\limits_{k=0}^\infty \lambda(D_n)\frac{1}{3^k}=\lambda(C_n)-\frac{3}{2}\lambda(D_n)>0.
\end{align*}
\textbf{Case 2:}
\textit{$\lambda\left(\{x\in X:\|f\|_\infty-|f(x)|>\epsilon\}\right)=\infty$ for every $\epsilon>0$.}\\
Then for every $n\in\mathbb{N}$, there exists $\epsilon_n>0$ such that $\epsilon_n>\epsilon_{n+1}$ and $\lim\limits_{n\to\infty}\epsilon_n=0$ with
\begin{equation*}
    \lambda\left(\{x\in X: \|f\|_\infty-|f(x)|\in(\epsilon_{n+1},\epsilon_n)\}\right)>0.
\end{equation*}
Hence setting
\begin{equation*}
    A:=\left\{x\in X: \|f\|_\infty-|f(x)|\in\left(\epsilon_{2n},\epsilon_{2n-1}\right),~n\in\mathbb{N}\right\},
\end{equation*}
and
\begin{equation*}
    B:=\left\{x\in X: \|f\|_\infty-|f(x)|\in\left(\epsilon_{2n+1},\epsilon_{2n}\right),~n\in\mathbb{N}\right\},
\end{equation*}
gives us the desired subsets of $X$.
\end{proof}
We now come to the characterization of smooth points in $L_\infty(X)$ but for that, we require the definition of a $\lambda$-atom.
\begin{definition}
    A subset $A\in\Sigma$ is called a \textit{$\lambda$-atom} if $\lambda(A)>0$ and 
    \begin{equation*}
        B\subset A,~\lambda(B)>0~\Rightarrow~B=A.
    \end{equation*}
\end{definition}
\begin{theorem}\label{smooth infinity 2}
    An element $f\in L_\infty(X)$ is smooth if and only if there exists a $\lambda$-atom $A$ such that $|f(x)|=\|f\|_\infty$ for almost every $x\in A$ and 
    \begin{equation*}
        \lambda\left(\{x\in X\backslash A:|f(x)|>\|f\|_\infty-\epsilon\}\right)=0,
    \end{equation*}
    for some $\epsilon>0$.
\end{theorem}
\begin{proof}
By Theorems \ref{smooth}, \ref{representation} and \ref{integral representation}, we have that $f\in L_\infty(X)$ is smooth if and only if there exists a unique $\lambda$-ultrafilter $\mathcal{U}$ such that 
\begin{equation*}
    \lim\limits_\mathcal{U}|f|=\|f\|_\infty.
\end{equation*}
We first prove the sufficiency. Set
\begin{equation*}
    \mathcal{V}_A:=\{B\in \Sigma:B\supset A\}.
\end{equation*}
Clearly, $\mathcal{V}_A$ is a $\lambda$-ultrafilter since no proper measurable subset of $A$ has nonzero measure. Clearly, $\lim\limits_{\mathcal{V}_A}|f|=\|f\|_\infty$. Suppose $\mathcal{U}$ is a $\lambda$-ultrafilter such that $\lim\limits_\mathcal{U}|f|=\|f\|_\infty.$ Let us assume that $\lim\limits_\mathcal{U}f=e^{i\theta}\|f\|_\infty$, for some $\theta\in[0,2\pi)$. Then clearly, 
\begin{equation*}
    A\cup\{x\in X\backslash A: |f(x)|>\|f\|_\infty-\epsilon\}\supset\left\{x\in X:\left|f(x)-e^{i\theta}\|f\|_\infty\right|<\epsilon \right\}\in\mathcal{U}.
\end{equation*}
Since $\{x\in X\backslash A: |f(x)|>\|f\|_\infty-\epsilon\}\notin\mathcal{U}$, by Lemma \ref{significant}, $A\in\mathcal{U}$. Hence clearly $\mathcal{V}_A\subset\mathcal{U}$ and thus $\mathcal{V}_A=\mathcal{U}$ since $\mathcal{V}_A$ is a $\lambda$-ultrafilter.\par
Conversely, suppose there is no $\lambda$-atom $A$ such that $|f(x)|=\|f\|_\infty$ for almost every $x\in A$. Then either
\begin{equation*}
    \lambda\left(\{x\in X: |f(x)|=\|f\|\}\right)=0,
\end{equation*}
or there exist $A$ and $B$ disjoint subsets in $\Sigma$ such that 
\begin{equation*}
    A\sqcup B\subset\{x\in X: |f(x)|=\|f\|_\infty\},
\end{equation*}
and $\lambda(A),\lambda(B)>0$. In the second case we consider 
\begin{align*}
    \mathcal{V}:=\{C\in\Sigma:C\supset A\},~~
    \mathcal{W}:=\{C\in\Sigma:C\supset B\}.
\end{align*}
Clearly $\mathcal{V}$ and $\mathcal{W}$ are contained in two distinct ultrafilters, say $\mathcal{V}'$ and $\mathcal{W}'$ and 
\begin{equation*}
    \lim\limits_{\mathcal{V}'}|f|=\lim\limits_{\mathcal{W}'}|f| =\|f\|_\infty.
\end{equation*}
In the first case, by Lemma \ref{partition}, there exist $A_1, A_2\subset X$ disjoint such that
\begin{equation*}
    \lambda\left(x\in A_i:\|f\|_\infty-|f(x)|<\epsilon\right)>0,
\end{equation*}
for every $\epsilon>0$ and $i=1,2$. Observe that 
\begin{equation*}
    \mathcal{B}_i:=\left\{\{x\in A_i:\|f\|_\infty-|f(x)|<\epsilon\}:\epsilon>0\right\},
\end{equation*}
is a $\lambda$-filter base for $i=1,2$ and hence in contained in a $\lambda$-ultrafilter $\mathcal{U}_i$. Also, clearly, $\mathcal{U}_1\neq\mathcal{U}_2$ and
\begin{equation*}
    \lim\limits_{\mathcal{U}_i}|f|=\|f\|_\infty,
\end{equation*}
for $i=1,2$.\\
Again if there exists a $\lambda$-atom $A$ with $|f(x)|=\|f\|_\infty$ for almost every $x\in A$ but
\begin{equation*}
    \lambda\left(\{x\in X\backslash A:|f(x)|>\|f\|_\infty-\epsilon\}\right)>0~~\textit{for every}~\epsilon>0,
\end{equation*}
then set
\begin{equation*}
    \mathcal{B}:=\left\{\{x\in X\backslash A:|f(x)|<\|f\|_\infty-\epsilon\}:\epsilon>0\right\}.
\end{equation*}
Clearly, $\mathcal{B}$ is a $\lambda$-filter base and is contained in some ultrafilter $\mathcal{U}$. Since $X\backslash A\in\mathcal{U}$, $\mathcal{U}$ and $\mathcal{V}_A$ are two distinct ultrafilters but 
\begin{equation*}
    \lim\limits_{\mathcal{U}}|f|=\lim\limits_{\mathcal{V}_A}|f|=\|f\|_\infty,
\end{equation*}
and hence the necessity.

\end{proof}
\subsection{Pointwise symmetry of Birkhoff-James orthogonality in $L_\infty(X)$}\hfill
\\

In this subsection we characterize the left symmetric and the right symmetric points of $L_\infty(X)$.
We begin with the characterization of the left symmetric points.
\begin{theorem}
    A non-zero $f\in L_\infty(X)$ is left symmetric if and only if $|f(x)|=\|f\|$ for almost every $x$ in some $\lambda$-atom $A$ and $f(x)=0$ for almost every $x\in X\backslash A$.
\end{theorem}
\begin{proof}
By Theorems \ref{smooth} and \ref{left}, if $f\in L_\infty(X)$ is left symmetric, $f$ is smooth and hence by Theorem \ref{smooth infinity 2}, there exists a $\lambda$-atom $A$, such that $|f(x)|=\|f\|_\infty$ for almost every $x\in A$. Further, by Theorems \ref{representation}, \ref{representation2} and \ref{integral representation}, $\lim\limits_\mathcal{U}f=0$ for every $\lambda$-ultrafilter $\mathcal{U}$ not containing $A$. Hence, for every $z\neq 0$, there exists $\epsilon_z>0$ such that 
\begin{equation*}
    \lambda\left(\{x\in X\backslash A:|f(x)-z|<\epsilon_z\}\right)=0.
\end{equation*}
Hence 
\begin{equation*}
    \left\{\{w\in\mathbb{K}:|w-z|<\epsilon_z\}:z\neq 0\right\},
\end{equation*}
is an open cover of $\mathbb{K}\backslash \{0\}$. Choose and fix a countable sub-cover given by:
\begin{equation*}
    \left\{\left\{w\in\mathbb{K}:\left|w-z_n\right|<\epsilon_{z_n}\right\}:z_n\neq 0,~n\in\mathbb{N}\right\}.
\end{equation*}
Hence 
\begin{equation*}
    \lambda\left(\{x\in X\backslash A:f(x)\neq 0\}\right)\leq\sum\limits_{n=1}^\infty\lambda\left(\{x\in X: \left|f(x)-z_n\right|\}\right)=0,
\end{equation*}
proving the necessity. The sufficiency follows easily from Theorems \ref{left}, \ref{representation}, \ref{integral representation} and \ref{integral and limit}.
\end{proof}
\begin{theorem}
    $f\in L_\infty(X)$ is right symmetric if and only if $|f(x)|=\|f\|_\infty$ for almost every $x\in X$.
\end{theorem}
\begin{proof}
By Theorems \ref{right}, \ref{representation}, \ref{integral representation} and \ref{integral and limit}, $f\in L_\infty(X)$ is right symmetric if and only if for every $z\in\mathbb{K}$ with $|z|<\|f\|_\infty$, there exists $\epsilon_z>0$ such that 
\begin{equation*}
    \lambda\left(\{x\in X:|f(x)-z|<\epsilon_z\}\right)=0.
\end{equation*}
Hence the sufficiency is easy to verify. For the necessity, observe that,
\begin{equation*}
    \left\{\{w\in\mathbb{K}:|w-z|<\epsilon_z\}:|z|<\|f\|_\infty\right\},
\end{equation*}
is an open cover of $\{z\in\mathbb{K}:|z|<\|f\|_\infty\}$. Choose and fix a countable sub-cover of the aforesaid cover given by:
\begin{equation*}
    \left\{\left\{w\in\mathbb{K}:\left|w-z_n\right|<\epsilon_{z_n}\right\}: \left|z_n\right|<\|f\|_\infty,~n\in\mathbb{N}\right\}.
\end{equation*}
Hence we obtain: 
\begin{align*}
        \lambda\left(\{x\in X:|f(x)|\neq\|f\|_\infty\right)&=\lambda\left(\{x\in X:|f(x)|<\|f\|_\infty\right\})\\
        &\leq\sum\limits_{n=1}^\infty \lambda\left(\left\{x\in X:\left|f(x)-z_n\right|<\epsilon_{z_n}\right\}\right)=0.
\end{align*}
\end{proof}

\section{Birkhoff-James orthogonality and its pointwise symmetry in $L_1$ spaces}
In this section, we first characterize Birkhoff-James orthogonality in $L_1(X)$ and then characterize smoothness and pointwise symmetry. As before, we assume the measure space to be $(X,\Sigma,\lambda)$. Our approach would be to characterize $J(f)$ for any non-zero $f\in L_1(X)$ and therefrom use the James characterization to characterize Birkhoff-James orthogonality. The characterizations of smoothness and pointwise symmetry would follow therefrom.\par
Since the dual of $L_1(X)$ is isometrically isomorphic to $L_\infty(X)$, we are going to assume that $L_\infty(X)$ is indeed the dual of $L_1(X)$ and any element $h\in L_\infty(X)$ acts on $L_1(X)$ as:
\begin{align*}
    f\mapsto \int\limits_Xh(x)f(x)d\lambda(x),~f\in L_1(X).
\end{align*}
\begin{lemma}\label{supp}
Suppose $f\in L_1(X)\setminus\{0\}$. Then for any $h\in L_\infty(X)=L_1(X)^*$, $h\in J(f)$ if and only if $h(x)=\overline{\sgn(f(x))}$ for almost every $x\in X$ such that $f(x)\neq0$, and $|h(x)|\leq1$ for almost every $x\in X$ such that $f(x)=0$.
\end{lemma}
\begin{proof}
The sufficiency follows from direct computation. For the necessity, note that
\begin{align*}
    \|f\|_1=\int\limits_Xh(x)f(x)d\lambda(x)\leq\int\limits_X\|h\|_\infty|f(x)|d\lambda(x)=\|f\|_1,
\end{align*}
whenever $h\in J(f)$. Hence from the condition of equality in the above inequality, we obtain $h(x)=\|h\|_\infty\overline{\sgn(f(x))}=\overline{\sgn(f(x))}$ for almost every $x\in X$, $f(x)\neq0$.
\end{proof}
From this lemma, we can now characterize Birkhoff-James orthogonality in $L_1(X)$.
\begin{theorem}\label{orth1}
Suppose $f,g\in L_1(X)$. Then $f\perp_Bg$ if and only if
\begin{align}\label{ortho1}
    \left|\int\limits_X\overline{\sgn(f(x))}g(x)d\lambda(x)\right|\leq\int\limits_{f(x)=0}|g(x)|d\lambda(x).
\end{align}
\end{theorem}
\begin{proof}
We first prove the necessity. Since $f\perp_Bg$, there exists $h\in L_\infty(X)$, such that $h\in J(f)$ and $\int\limits_Xg(x)h(x)d\lambda(x)=0$. By Lemma \ref{supp}, we now conclude
\begin{align*}
    \left|\int\limits_X\overline{\sgn(f(x))}g(x)d\lambda(x)\right|&=\left|\int\limits_{f(x)\neq0}\overline{\sgn(f(x))}g(x)d\lambda(x)\right|\\&=\left|\int\limits_{f(x)=0}g(x)h(x)d\lambda(x)\right|\leq\int\limits_{f(x)=0}|g(x)|d\lambda.
\end{align*}
Again, if \eqref{ortho1} holds, set:
\begin{align*}
    c=\frac{-\left|\int\limits_X\overline{\sgn(f(x))}g(x)d\lambda(x)\right|}{\int\limits_{f(x)=0}|g(x)|d\lambda(x)}.
\end{align*}
Consider $h:X\to\mathbb{K}$ given by
\begin{align*}
    h(x):=
    \begin{cases}
        \overline{\sgn(f(x))},~~&f(x)\neq0,\\
        c\,\overline{\sgn(g(x))},~~&f(x)=0.
    \end{cases}
\end{align*}
Clearly, $h\in L_\infty(X)$ and $h\in J(f)$ by Lemma \ref{supp}. But clearly $\int\limits_Xg(x)h(x)d\lambda(x)=0$, establishing the sufficiency.
\end{proof}
We now characterize the smooth points of $L_1(X)$.
\begin{theorem}\label{s}
$f\in L_1(X)$ is a smooth point if and only if $f\neq0$ almost everywhere on $X$.
\end{theorem}
\begin{proof}
If $f\neq0$ almost everywhere on $X$, $\int\limits_{f(x)=0}|g(x)|d\lambda(x)=0$ for every $g\in L_1(X)$ giving $f\perp_Bg$ if and only if
\begin{align*}
    \int\limits_X\overline{\sgn(f(x))}g(x)d\lambda(x)=0.
\end{align*}
Hence $f\perp_Bg$ and $f\perp_Bh$ for $g,h\in L_1(X)$ forces $f\perp_B(g+h)$, proving the sufficiency. To prove the necessity, assume $\lambda(\{x\in X:f(x)\neq0\})>0$. Consider $h_0,h_1:X\to\mathbb{K}$ given by
\begin{align*}
    h_i(x):=
    \begin{cases}
        \overline{\sgn(f(x))},~~&f(x)\neq0,\\
        i,~~&f(x)=0,
    \end{cases}
\end{align*}
for $i=0,1$. Then by Lemma \ref{supp}, $h_0$ and $h_1$ are two distinct support functionals of $f$ and hence $f$ cannot be a smooth point of $L_1(X)$. 
\end{proof}
We now characterize the pointwise symmetry of Birkhoff-James orthogonality in $L_1(X)$. We first address the left-symmetric case.
\begin{theorem}
$f\in L_1(X)$ is a left-symmetric point if and only if exactly one of the following conditions holds:
\begin{enumerate}
    \item $f\equiv0$.
    \item $f\not\equiv0$ and $\Sigma=\{\emptyset, X\}$.
    \item There exist disjoint $\Sigma$-atoms $A$ and $B$ such that $A\sqcup B=X$ and $\lambda(A)|f(x)|=\lambda(B)|f(y)|$ for almost every $x\in A$ and $y\in B$.
\end{enumerate}
\end{theorem}
\begin{proof}
The sufficiency can be obtained from Theorem \ref{orth1} by direct computation. For, the necessity, let $f\in L_1(X)$, $f\not\equiv0$. We consider two cases:\\
\textbf{Case I:} $\lambda(\{x\in X:f(x)=0\})>0$.\\
Set $A\subseteq\{x\in X:f(x)=0\}$ such that $\infty>\lambda(A)>0$. Consider $g:X\to\mathbb{K}$ given by
\begin{align*}
    g(x):=
    \begin{cases}
        f(x),~~&f(x)\neq0,\\
        \frac{\|f\|_1}{\lambda(A)},~~&x\in A,\\
        0,~~&\text{otherwise}.
    \end{cases}
\end{align*}
Clearly, $g\in L_1(X)$ and by Theorem \ref{orth1}, $f\perp_Bg$ but $g\not\perp_Bf$.\\
\textbf{Case II:} $\lambda(\{x\in X:f(x)=0\})=0$.\\
Since $f\in L_1(X)$ is a smooth point, from the proof of Theorem \ref{s}, for any $g\in L_1(X)$, $f\perp_Bg$ if and only if 
\begin{align*}
    \int\limits_X\overline{\sgn(f(x))}g(x)d\lambda(x)=0.
\end{align*}
Suppose, there do not exist disjoint $\Sigma$-atoms $A$ and $B$ such that $A\sqcup B=X$ and $\lambda(A)|f(x)|=\lambda(B)|f(y)|$ for almost every $x\in X$ and $y\in Y$. Then there exists $A\in \Sigma$ such that
\begin{align*}
    0<\int\limits_A|f(x)|d\lambda(x)<\int\limits_{X\setminus A}|f(x)|d\lambda(x).
\end{align*}
Set $\alpha=\int\limits_A|f(x)|d\lambda(x)$ and $\beta=\int\limits_{X\setminus A}|f(x)|d\lambda(x)$. Now, $g:X\to\mathbb{K}$ given by
\begin{align*}
    g(x):=
    \begin{cases}
        \beta f(x),~~&x\in A,\\
        -\alpha f(x),~~&x\notin A,
    \end{cases}
\end{align*}
is a smooth point of $L_1(X)$ by Theorem \ref{s}. Hence for any $h\in L_1(X)$, $g\perp_Bh$ if and only if
\begin{align*}
    \int\limits_X\overline{\sgn(g(x))}h(x)d\lambda(x)&=0\\
    \Leftrightarrow \int\limits_A\overline{\sgn(f(x))}h(x)d\lambda(x)-\int\limits_{X\setminus A}&\overline{\sgn(f(x))}h(x)=0.
\end{align*}
Hence by our choice of $A\in \Sigma$, $f\perp_Bg$ but $g\not\perp_Bf$.

\end{proof}
We conclude this section with the characterization of the right-symmetric points of $L_1(X)$.
\begin{theorem}
A non-zero function $f\in L_1(X)$ is right-symmetric if and only if $\{x\in X:f(x)\neq0\}$ is a $\Sigma$-atom.
\end{theorem}
\begin{proof}
Clearly, if $A=\{x\in X:f(x)\neq0\}$ is a $\Sigma$-atom, then by Theorem \ref{orth1}, for any $g\in L_1(X)$, $g\perp_Bf$ if and only if $g|_A\equiv0$, Hence $f\perp_Bg$ if $g\perp_Bf$.\par
Conversely, if there exist disjoint measurable subsets $A$ and $B$ of finite positive measure such that $A\sqcup B\subseteq\{x\in X:f(x)\neq0\}$, then without loss of generality, we assume 
\begin{align*}
    0<\int\limits_A|f|d\lambda\leq\int\limits_B|f|d\lambda.
\end{align*}
Setting $g: X\to\mathbb{K}$ given by
\begin{equation*}
    g(x):=
    \begin{cases}
        \sgn(f(x)),~~&x\in A,\\
        0,~~&x\notin A,
    \end{cases}
\end{equation*}
we get $g\in L_1(X)$. Also, by Theorem \ref{orth1}, $g\perp_Bf$ and $f\not\perp_Bg$.
\end{proof}

\section{Birkhoff-James orthogonality and its pointwise symmetry in $L_p(X)$, $p\in(1,\infty)\setminus\{2\}$}
In this section, we characterize Birkhoff-James orthogonality and its pointwise symmetry in $L_p(X)$ for $1<p<\infty$, $p\neq2$. It is well-known that $L_p(X)$ is smooth and hence the characterization of smoothness here is redundant. Our approach for $L_p(X)$, $p\in(1,\infty)\setminus\{2\}$ is similar to the $L_1(X)$ case. We first study the support functional (which is unique here as the space is smooth) of a non-zero element and therefrom obtain a characterization of Birkhoff-James orthogonality by James' characterization. The characterization of pointwise symmetry would then follow from the orthogonality characterization.\par
Let us fix $p\in(1,\infty)\setminus\{2\}$. The following theorem characterizing the (unique) support functional of $f\in L_p(X)\setminus\{0\}$ follows directly from the condition of equality in Hölder's inequality.
\begin{theorem}\label{suppp}
Let $f\in L_p(X)\setminus\{0\}$ and let $\frac{1}{p}+\frac{1}{q}=1$. Suppose $g\in L_q(X)=L_p(X)^*$. Then $g\in J(f)$ if and only if
\begin{equation*}
    g(x)=\frac{1}{\|f\|_p^{p-1}}\overline{\sgn(f(x))}|f(x)|^{p-1},~~x\in X.
\end{equation*}
\end{theorem}
Using this result, we now characterize Birkhoff-James orthogonality in $L_p(X)$.
\begin{theorem}\label{orthp}
If $f,g\in L_p(X)$, then $f\perp_Bg$ if and only if
\begin{align*}
    \int\limits_X\overline{\sgn(f(X))}|f(x)|^{p-1}g(x)d\lambda(x)=0.
\end{align*}
\end{theorem}
We can now characterize pointwise symmetry of Birkhoff-James orthogonality in $L_p(X)$.
\begin{theorem}
Suppose $f\in L_p(X)$. Then $f$ is left-symmetric if and only if $f$ is right-symmetric if and only if exactly one of the following conditions holds:
\begin{enumerate}
    \item $f\equiv0$.
    \item $\{x\in X:f(x)\neq0\}$ is a $\Sigma$-atom.
    \item There exist $\Sigma$-atoms $A$ and $B$ such that $\{x\in X:f(x)\neq0\}=A\sqcup B$ and $\lambda(A)|f(x)|^p=\lambda(B)|f(y)|^p$ for almost every $x\in A$ and $y\in B$.
\end{enumerate}
\end{theorem}
\begin{proof}
The sufficiency can be obtained from Theorem \ref{orthp} by an elementary computation. For the necessity, let us assume that $f\not\equiv 0$ and $\{x\in X:f(x)\neq0\}$ is not a $\Sigma$-atom. We consider the following two cases:\\
\textbf{Case I:} There exist $\Sigma$-atoms $A$ and $B$ such that $\{x\in X:f(x)\neq0\}=A\sqcup B$.\\
If $g\in L_p(X)$ with $\{x\in X:g(X)\neq0\}=A\sqcup B$ such that $g\perp_Bf$ and $f\perp_Bg$, then for almost every $x\in A$ and $y\in B$,
\begin{equation}\label{condio1}
    g(x)=\frac{\lambda(B)|f(y)|^{p-1}\overline{\sgn(f(y))}}{\lambda(A)|f(x)|^{p-1}\overline{\sgn(f(x))}}g(y),
\end{equation}
and
\begin{equation}\label{condio2}
    f(x)=\frac{\lambda(B)|g(y)|^{p-1}\overline{\sgn(g(y))}}{\lambda(A)|g(x)|^{p-1}\overline{\sgn(g(x))}}f(y).
\end{equation}
Hence
\begin{align*}
    \left[\lambda(B)|f(y)|^p\right]^{p-2}=\left[\lambda(A)|f(x)|^p\right]^{p-2}.
\end{align*}
Since $p\neq2$, $f$ must satisfy condition 3. However, using \eqref{condio1} or \eqref{condio2}, we can always construct $g\in L_p(X)$ with $\{x\in X:g(x)\neq0\}=A\sqcup B$ such that $f\perp_Bg$ or $g\perp_Bf$ respectively.\\
\textbf{Case II:} There exist $A,B,C\in\Sigma$ disjoint such that all the sets are of finite positive measure and $A\sqcup B\sqcup C\subseteq\{x\in X:f(x)\neq0\}$. \\
Without loss of generality, let us assume that 
\begin{align*}
    0<\int\limits_{A}|f(x)|^pd\lambda(x)<\int\limits_{B\sqcup C}|f(x)|^pd\lambda(x).
\end{align*}
Consider $g_{a,b}:X\to\mathbb{K}$ given by 
\begin{align*}
    g_{a,b}(x):=
    \begin{cases}
        a f(x),~~&x\in A,\\
        b f(x),~~&x\in B\sqcup C,\\
        0,~~&\text{otherwise},
    \end{cases}
\end{align*}
for some $a,b\in\mathbb{K}$. Then,
\begin{align*}
    \int\limits_{X}\overline{\sgn(f(x))}|f(x)|^{p-1}g_{a,b}(x)d\lambda(x)=a\int\limits_A|f(x)|^pd\lambda(x)+b\int\limits_B|f(x)|^pd\lambda(x),
\end{align*}
and
\begin{align*}
    \int\limits_{X}\overline{\sgn(g_{a,b}(x))}|g_{a,b}(x)|^{p-1}f(x)d\lambda(x)&=\overline{\sgn(a)}|a|^{p-1}\int\limits_A|f(x)|^pd\lambda(x)\\
    &+\overline{\sgn(b)}|b|^{p-1}\int\limits_B|f(x)|^pd\lambda(x)
\end{align*}
Thus, $f\perp_Bg_{a,b}$ but $g_{a,b}\not\perp_Bf$ when
\begin{align*}
    a=\int\limits_{B\cup C}|f(x)|^pd\lambda(x),~~b=-\int\limits_{A}|f(x)|^pd\lambda(x),
\end{align*}
and $g_{a,b}\perp_Bf$ but $f\not\perp_Bg_{a,b}$ when
\begin{align*}
    a=\left[\int\limits_{B\sqcup C}|f(x)|^pd\lambda(x)\right]^\frac{1}{p-1},~~b=-\left[\int\limits_{A}|f(x)|^pd\lambda(x)\right]^\frac{1}{p-1}.
\end{align*}
\end{proof}

\end{document}